\documentclass[12pt]{amsart}

\usepackage{amssymb, amsthm, amsmath, graphicx, yfonts, wasysym, appendix, url, verbatim}
\usepackage[margin=3cm]{geometry}
\usepackage{xcolor}
\usepackage{amssymb}
\usepackage{amscd}
\usepackage{hyperref}
\usepackage{mathtools}

\let\phi\varphi

\newcommand{\IP}[2]{\left< #1 , #2 \right>}

\newcommand{\cH}{\mathcal{H}}

\newcommand{\cc}{\boldsymbol{c}}

\newcommand{\ol}[1]{\overline{#1}}
\newcommand{\go}{\mathring{g}}
\newcommand{\gp}{g^+}

\newcommand{\om}{\omega}

\newcommand{\Ric}{\operatorname{Ric}}

\newcommand{\R}{\ensuremath{\mathbb{R}}}

\newcommand{\rpl}                         
{\mbox{$
\begin{picture}(12.7,8)(-.5,-1)
\put(0,0.2){$+$}
\put(4.2,2.8){\oval(8,8)[r]}
\end{picture}$}}

\numberwithin{equation}{section}

\newcounter{theorem}
\newtheorem{thm}[theorem]{Theorem}
\newtheorem*{thm*}{Theorem }
\newtheorem{lemma}[theorem]{Lemma}
\newtheorem{prop}[theorem]{Proposition}
\newtheorem{cor}[theorem]{Corollary}
\newtheorem*{lemma*}{Lemma \thesubsection}
\newtheorem*{prop*}{Proposition}
\newtheorem*{cor*}{Corollary \thesubsection}
\newtheorem{problem}{Problem}

\theoremstyle{definition}
\newtheorem{definition}[theorem]{Definition}
\newtheorem*{definition*}{Definition \thesubsection}
\newtheorem{example}[theorem]{Example}
\newtheorem*{example*}{Example \thesubsection}
\theoremstyle{remark}
\newtheorem{remark}[theorem]{Remark}
\newtheorem*{remark*}{Remark \thesubsection}

\newtheorem{innercustomthm}{{\bf Theorem}}
\newenvironment{customthm}[1]
  {\renewcommand\theinnercustomthm{#1}\innercustomthm}
  {\endinnercustomthm}

\newtheorem{innercustomprop}{{\bf Proposition}}
\newenvironment{customprop}[1]
  {\renewcommand\theinnercustomprop{#1}\innercustomprop}
  {\endinnercustomprop}

\def\sideremark#1{\ifvmode\leavevmode\fi\vadjust{\vbox to0pt{\vss
 \hbox to 0pt{\hskip\hsize\hskip1em
 \vbox{\hsize2.5cm\tiny\raggedright\pretolerance10000
  \noindent #1\hfill}\hss}\vbox to8pt{\vfil}\vss}}}%

\begin{document}
\renewcommand{\today}{}
\title{CMC Foliations and their conformal aspects}

\author{A.\ Rod Gover and Valentina-Mira Wheeler}
\address{A.R.G.:Department of Mathematics\\
  The University of Auckland\\
  Private Bag 92019\\
  Auckland 1142\\
  New Zealand}
\email{r.gover@auckland.ac.nz}

\address{V.-M. Wheeler\\
School of mathematics and applied statistics \\
           University of Wollongong\\
           Northfields Avenue\\
           Wollongong, NSW, 2500\\
           Australia\\}
\email{vwheeler@uow.edu.au}

\begin{abstract}
  On a manifold we term a hypersurface foliation a slicing if it is
  the level set foliation of a slice function -- meaning some real
  valued function $f$ satisfying that $df$ is nowhere zero. On
  Riemannian manifolds we give a non-linear PDE on functions whose
  solutions are generic constant-mean-curvature (CMC) slice
  functions. Conversely, to any generic transversely-oriented
  constant-mean-curvature foliation the equation uniquely associates
  such a function. In one sense the equation is a scalar analogue of
  the Einstein equations.  Given any slicing we show that, locally,
  one can conformally prescribe any smooth mean curvature function. We
  use this to show that, locally on a Riemannian manifold, a slicing
  is CMC for a conformally related metric. These results admit global
  versions assuming certain restrictions.  Finally, given a
  conformally compact manifold we study the problem of normalising the
  defining function so that it is a CMC slice function for a
  compactifying metric. We show that two cases of this problem are
  formally solvable to all orders.
\end{abstract}

\maketitle

\subjclass{MSC2020: Primary 53C12, 53C18, 53A10;
  Secondary 58J90, 53B20, 53A05}

\section{Introduction}

 Recall on an $n$-manifold $M$ (here and throughout we take $n\geq 2$)
 a foliation of dimension $p$ is a rank $p$ distribution $\mathcal{H}$
 that is integrable
 \cite{frobenius1877ueber,lawson1974foliations}. Here we consider only
 the case of $p=n-1$ so that $\mathcal{H}$ is a {\em hyperplane
   distribution}, meaning that it is a corank 1 vector subbundle of
 the tangent bundle $TM$, and integrability means that for vector
 fields $X,Y\in \Gamma(\cH)$  the Lie bracket $[X,Y]$ also
 lies in $\Gamma(\cH)$. In this case each leaf is a {\em
   hypersurface}, meaning an embedded submanifold of codimension one.
 From the distribution $\cH$ we have an exact sequence
$$
0\to \cH\to TM \to Q\to 0 ,
$$ and so dually a rank 1 subbundle $Q^*$ of $T^*M$ consisting of
annihilators of $\cH$.
The foliation is
{\em transversely orientable} if $Q^*$ admits a global section, and the choice of
any such a section determines a transverse orientation \cite{moerdijk2003introduction}.

In terms of $Q^*$ the
condition of integrability is that any local section $\nu\in
\Gamma(Q^*)$ satisfies that its exterior derivative $d\nu$ lies in the
ideal of forms generated by $\nu$, equivalently $\nu\wedge d\nu=0$. This
means that on contractible open sets $U$ (so we will say {\em
  locally}) the leaf space is given by the level sets of a function
$f:U\to \R $ with $df\in \Gamma_U (Q^*)$ nowhere zero on $U$.
Since such functions always exist locally, and may exist globally, it will be useful to distinguish and discuss foliations given by such
functions.
\begin{definition}\label{slicing-def}
  On any dimension $n$ submanifold $N$ of $M$ we will term a
  hypersurface foliation a {\em slicing} if it is the level set
  foliation of some smooth function $f:N\to \R$ satisfying that $df$ is nowhere zero
  on $N$. Any such function will be called a {\em slice function}
    for the slicing.
\end{definition}
\noindent Given a slice function $f$, then $df$ assigns an orientation
to the slicing determined by $f$. Of course a closed manifold cannot
admit a global slicing. On the other hand, as is well known, closed
manifolds do admit Morse functions. On a closed manifold a Morse
function has a finite set of critical points, and so is a slice
function on the complement of those points. Similarly Morse-Bott
functions restrict to slice functions on an open dense subset.

Unless otherwise stated, all functions and structures are taken to be
smooth, that is $C^\infty$ on their domain.

From Section \ref{fol-sec} we work in the Riemannian setting  (and see there  for the notation used below) to find a
partial differential equation that characterises constant mean
curvature (CMC) foliations that are {\em generic} in the sense that
the mean curvature function $H_f$ has no critical points.
We obtain
the following result.

  \begin{thm}\label{cmc}
 A  foliation is  generic CMC iff it admits a  slice function $f$ satisfying
\begin{equation}\label{CMCeq}
  (n-1)\lambda f= \frac{1}{| df|} \Delta f- \frac{1}{|df|^3}\cdot (\nabla^a f) (\nabla^b f)(\nabla_a\nabla _b f) ,
\end{equation}
where $\lambda$ is either $+1$ or $-1$.
\end{thm}

In Riemannian geometry the Einstein equations are a non-linear PDE system that
equate a non-linear rational function of the derivatives of the
metric components $g_{ij}$ to a constant multiple of these: $\Ric_g=\mu g$,
where $\Ric_g$ is the Ricci curvature of the metric $g$. The equation \ref{CMCeq}  is evidentally a scalar equation in the same spirit, it states
\begin{equation}\label{the-eq}
H_f= \lambda f ,
\end{equation}
where $H_f$ is the mean curvature function of the slicing $f$, and this is a non-linear rational function of the derivatives of $f$.

 Section \ref{conf-sec} moves into conformal aspects of the
 problem. Here, given a slicing, we consider changing the metric
 conformally, i.e. replacing $g$ with $\widehat{g}=e^{2\om} g$
 ($\om\in C^\infty(M)$), so that with respect to the new metric the
 slicing is CMC. We start the simpler problem of conformally finding a metric that makes it minimal (cf.\ \cite{rummler1979quelques,haefliger1980some}).
 \begin{prop}\label{min-prop}
  Given a slicing $f$ on a Riemannian manifold
    $(M,g)$, there is locally a conformally related metric
    $\widehat{g}=e^{2\om}g$ so that the slicing is minimal with
  respect to $\widehat{g}$.

  This extends to a global result if the manifold is contractible or it is  a
  product compatible with the slicing.
  \end{prop}
\noindent A slicing determines, in an obvious way, a diffeomorphism
from collars around the slicing leaves to products of a typical leaf
with an interval in $\mathbb{R}$, see the proof of this proposition in Section \ref{conf-sec}. This may extend to give a product
structure on the entire manifold and that is what we mean by the last
part of the proposition statement. Prescription of minimal curvature via general metric changes has been studied in many places, see e.g. \cite{rummler1979quelques,haefliger1980some}.
Here we focus on what can be achieved by simple conformal theory.

The result above generalises, as follows.

\begin{thm}\label{conf-presc}
 Let $f$ be a slice function on a  Riemannian
   manifold $(M,g)$.
    Locally,  we can conformally prescribe the mean curvature to be any smooth
    function.
That is, given a smooth
    function $h$ on $M$, there is a metric
   $\widehat{g}$ in the conformal class $[g]$ such that the equation
  \begin{equation}\label{p-eq}
h=H^{\widehat{g}}_f
  \end{equation}
  is satisfied.

This extends to a global result for smooth functions $h$ with compact
support if the manifold is contractible, or it is a product compatible
with the slicing.
  \end{thm}

In \cite{walczak1984mean} Walczak looks at the type of functions that
can arise as the mean curvature of a Riemannian manifold with a
transversely orientable foliation. Part of the treatment there uses
conformal transformations, and this aspect has  links to our
result. Walczak and Schweitzer \cite{schweitzer2004prescribing}
consider foliations of general codimension. Given a vector field on
the manifold, they provide conditions under which the vector field
becomes the mean curvature vector of the foliation with respect to
some Riemannian metric on the manifold.

A main point here is that Theorem \ref{conf-presc} enables us to obtain generic CMC foliations as follows.
\begin{cor}\label{main-cor}
 Given a slice function $f$ and $\lambda:=\pm 1$, locally there is $\widehat{g}$, conformally related to $g$, such that
  $$
H^{\widehat{g}}_f =\lambda f .
$$
So the slicing determined by $f$ solves equation (\ref{CMCeq}) for the metric
$\widehat{g}\in [g]$, and in particular is CMC for $\widehat{g}$.
  \end{cor}
 \noindent The full and stronger result is stated in Corollary
 \ref{c-solve}.  If one drops the requirement to solve conformally
 then one can use the approach of \cite{walczak1984mean} to find
 metrics solving \eqref{the-eq}.

Aside from solving equation (\ref{CMCeq}), and the $\lambda=0$
version, both Proposition \ref{min-prop} and Corollary \ref{main-cor}
provide a way, in the presence of a foliation, of choosing
distinguished metrics from a conformal class.

\medskip

In Section 6, we look at conformally compact manifolds. These are
structures that have been the subject of intense scrutiny in both
mathematics and physics since the pioneering works of Fefferman and
Graham \cite{graham1985charles,fefferman2012ambient}, see
e.g. \cite{mazzeo1988hodge,maldacena1999large,Graham2003scattering,case2021boundary,fine2022ambient,graham2019chern} and references therein.

Let $\ol{M}$ be a compact $n$-manifold with boundary, and
write $M$ for the interior. So $\ol{M}=M\cup \partial M$ where the
boundary $\partial{M}$ is a smooth closed $(n-1)$-manifold. A metric
$\gp$ on $M$ is said to be {\em conformally compact} if the following
holds: there is a metric $g$ on $\ol{M}$ such that, in some collar
neighbourhood of $\partial M$, we have $g=r^2 \gp$ with $r$ a slice
function such that its zero locus is exactly $\partial M$, i.e.,
$$
\partial M = \mathcal{Z}(r).
$$
We say that such a slice function is (boundary) {\em defining}.
The restriction of such $g$ to $\otimes^2 T\partial M$ determines a
canonical conformal structure on $\partial M$. This link between Riemannian and conformal geometry is one of key motivations for studying conformally compact manifolds.

It is interesting to consider whether, given $g^{+}$, there is some
canonical way to choose $r$, at least given a metric from the
conformal class induced on $\partial M$. A conformally compact metric
$\gp$ is said to be {\em asymptotically hyperbolic (AH)} if $|dr|_g=1$
{\em along} $\partial M$. This means that the sectional curvatures of
$\gp$ approach $-1$ asymptotically to the boundary.  For such metrics
Graham and Lee show that, in a neighbourhood of the boundary, there is
a slice function $r$ so that $|dr|_g=1$ \cite{graham1991einstein}.
Linked to our theme of CMC
foliations a very natural question concerns the following possible
alternative geometric characterisation of distinguished defining
functions: \\

\begin{problem}\label{question}
 Given $\gp$, is there a defining slice function $\rho$ so that
$$
g:=\rho^2\gp
$$
is a metric to the boundary and the level sets of $\rho$ are CMC for $g$?
\end{problem}

Studying versions of this question is the main aim of  Section \ref{ccpct}.
A simpler question is whether, given a conformal class $\cc$ on
$\ol{M}$ (meaning $\cc$ is a equivalence class of conformally related Riemannian metrics), there are compatible examples of such $g$ and $\gp$. Of course the upper half space and Poincar\'e-ball models of hyperbolic space provide examples, see Examples \ref{Eucl-half} and \ref{Poin-disk}.  The
following Proposition shows that examples exist far more generally.

\begin{prop}\label{step}
Let $\ol{M}$ be a compact $n$-manifold with boundary that is equipped with a
conformal structure $\cc$. There is a conformally compact metric $\gp$
on $M$ such that in some collar
$$
g=\rho^2 \gp
$$ where $g\in \cc$ is a metric to the boundary, $\rho$ is a slice
function with mean curvature $\rho=H^g_\rho$, and $\partial M=\mathcal{Z}(\rho)$.
\end{prop}
\noindent This is proved in Section \ref{ccpct}. Returning to the
Question above, we seek first defining functions $r$ such that $H^g_r$
is minimal. We show that for AH manifolds the problem can be solved formally.

\begin{prop}\label{inductionminimal}
Let $\ol{M}$ compact $n$-manifold with boundary, and $\gp$ an
asymptotic hyperbolic metric on the interior.

For each boundary metric $g_{\partial M}$, in the canonical conformal
class, there is a slice function $\bar{r}$, that is defining for the boundary $\partial
M$, such that $\bar{g} =(\bar{r})^2 \gp $ induces the metric
$g_{\partial M}$ and the foliation
determined by $\bar{r}$ satisfies
    $$
    H^{\bar{g}}_{\bar{r}}
    =O({\bar{r}}^\ell)
    $$
    for any $\ell\in \mathbb{Z}_{\geq 0}$.
    This determines $\bar{r}$ uniquely up to $+O({\bar{r}^{\ell+1}})$.
\end{prop}
\noindent Here and throughout $O(\rho)$ means a function of the form $\rho F$
where $F$ is smooth on $\ol{M}$.

For the study of conformal hypersurface invariants it can be useful to
use the hypersurface to somehow normalise or capture the transverse
jets of the choices of metric from the conformal class. For example in
\cite{blitz2022toward,blitz2023dirichlet}
the $T$-curvatures of
\cite{gover2021conformal} are combined with a formal solution of a
singular-Yamabe problem (which formally finds a canonical conformally
compact $\gp$ along the hypersurface) to achieve this. It seems that
Proposition \ref{inductionminimal} could provide an effective
alternative to the use of the $T$-curvatures for certain
applications. Indeed it would interesting to understand the links of
the results here with the $T$-curvatures. Each $T$-curvature is a type
of higher order mean curvature.

More generally (than Proposition \ref{inductionminimal}) one might
hope to be able to find a slice function $r$ such that $H^g_r$ is a
function of $r$, $G(r)$. (See the discussion in Section \ref{ccpct}.)
We show that, at least formally to all orders, $H^g_r=r$ is
achievable, as follows.
\begin{prop}\label{inductionCMC}
Let $\ol{M}$ compact $n$-manifold with boundary, and $\gp$ an
asymptotic hyperbolic metric on the interior.

For each boundary metric $g_{\partial M}$, in the canonical conformal
class, there is a slice function $\bar{r}$, that is defining for the boundary $\partial
M$, such that $\bar{g} =(\bar{r})^2 \gp $ induces the metric
$g_{\partial M}$ and the foliation
determined by $\bar{r}$ satisfies
    $$
    H^{\bar{g}}_{\bar{r}}
    =\bar{r}+O({\bar{r}}^\ell)
    $$
    for any $\ell\in \mathbb{Z}_{\geq 0}$.
This determines $\bar{r}$ uniquely up to $+O({\bar{r}}^{\ell+1})$.
\end{prop}
It seems likely that our approach to Propositions
\ref{inductionminimal} and \ref{inductionCMC} could be adapted to drop
the AH restriction and treat general conformally compact manifolds.

Rather than the question above, in \cite{mazzeo2011constant} Mazzeo
and Pacard treat the problem of foliations near the infinity of
AH metrics that are, in the notation here, CMC for
the metric $g^+$. In particular Theorem 1.1 of that source provides a
main result with cases depending on the Yamabe constant of the
boundary metric. Although there are similarities to the question
looked here there are also critical differences which means that one expects rather different results, see Remark \ref{cfMP}.  They point out that much of their work extends to
other Weingarten foliations.  A similar remark applies to Sections
\ref{fol-sec} and \ref{CMCfol} below. The ideas there could easily
applied to other Weingarten curvatures (i.e.\ $k$-homogenous
functionals of the principal curvatures).
Existence of CMC foliations on high order perturbations of the AdS-Schwarzschild space was proved by \cite{rigger2004foliation}, using mean curvature flow. Neves and Tian \cite{neves2009existence,neves2010existence}  established
uniqueness and extended the existence theory in this setting.

\section{Basics} \label{basics}

Given a slicing, different functions can provide the same foliation, as follows.
\begin{lemma}\label{slice-lemma}
 On a manifold $N$, given a slice function $f$ taking values in
 $I\subset \mathbb{R}$, and a strictly increasing (decreasing)
 smooth function $F:I\to \mathbb{R}$ then $F\circ f$ is another slice
 function for the given foliation with the same (respectively,
 opposite) orientation. The converse also holds: For a slice function $f$, if $f:N\to I$ is
 surjective then any other slice function for the foliation with the same (respectively,
 opposite) orientation is $F\circ f$ for a strictly increasing
 (decreasing) smooth function $F:I\to \mathbb{R}$.
\end{lemma}
\begin{proof}
  $\Rightarrow $:   This is obvious as $F\circ f$ has nowhere vanishing derivative and
  has the same  level sets as $f$, with the orientation in agreement or swapped depending on whether $F$ is strictly increasing or, respectively, decreasing:
  $$
  d (F\circ f)_p= F'(f)(p)\cdot df_p \qquad p\in N.
  $$
  If  $F'(f)(p)>0$ then $F'(f)>0$ at all points in $N$, and
the other case is similar.

$\Leftarrow$: Suppose that $h$ is another slice function for the slicing given by $f$. Then
$$
dh = k df ,
$$
for some real valued  function $k$.
As $f$ is a slice function, then given any point $p$
there is an open neighbourhood $U$ of $p$ on which there are a coordinates
$(x^1,\cdots ,x^n)$ with $x^1=f$.
Then taking the exterior derivative of the previous display gives
$$
dk\wedge dx^1=0 ,
$$
and so
$$
\frac{\partial k}{\partial x^i}=0 \qquad i=2,\cdots ,n
$$ and $k =k(x^1)$, and so also we have $h=F(x^1)$ for some 1-variable
function $F$ satisfying that $F'$ is not zero on the range of
$f$. Since this holds locally everywhere on $N$ the result
follows.  \end{proof}

\medskip

In the following we will work on a manifold $M$ and usually assume this
admits a slicing, so this will imply topological restrictions on
$M$. The point is that the results have local implications on any
manifold.

\section{PDEs defining slicings by CMC hypersurfaces}\label{fol-sec}

In the following we shall work on a connected Riemannian manifold $(M,g)$ of
dimension $n\geq 2$. Some notation: Given a 1-form field $u$ we write $|u|$ (or sometimes $|u|_g$) to mean $\sqrt{g^{-1}(u,u)}$. We write $\nabla$ or $\nabla_a$ to denote the Levi-Civita connection determined by $g$ and write $\Delta $ for the Laplacian defined by
$$
\Delta:= g^{ab}\nabla_a\nabla_b =\nabla^b\nabla_b .
$$

Recall that along an embedded
hypersurface $\Sigma$ with unit conormal field $\nu_a$ one has that the
second fundamental form is given by
$$
h_{ab}= (\delta_a^c-\nu_a\nu^c)\nabla_c \nu_b= \nabla_a \nu_b -\nu_a\nu^c\nabla_c \nu_b ,
$$ and this is independent of how $\nu_a$ is smoothly extended
off $\Sigma$. The indices refer to the ambient $TM$. Here and below indices are abstract unless otherwise indicated. Thus the
mean curvature is given by
\begin{equation}\label{mean-ex}
H= \frac{1}{(n-1)} (\nabla_a \nu^a - \nu^a\nu^b \nabla_a \nu_b) \qquad \nu^b\nu_b =1 \quad \mbox{along} \quad \Sigma,
\end{equation}
and again this is independent of how $\nu$ is extended of $\Sigma$.  A
hypersurface is said to be of {\em constant mean curvature (CMC)} if $H$ is
constant along the hypersurface.

We have the following observation.
Let $f:M\to \mathbb{R}$ be a smooth slice function. Then the
unit 1-form field $\nu_f:=df /|df|$ (which we will usually denote simply by $\nu$ if $f$ is understood) is everywhere normal to the slicing
of $M$ by level sets of $f$.
Now substitute $\nu=df /|df|$ into expression (\ref{mean-ex}) to give
\begin{equation}\label{firstH}
H= \frac{1}{(n-1)} (\nabla_a \nu^a) ,
\end{equation}

and,  using that
$$
\nabla_a \frac{1}{| df|}=\nabla_a (\nabla^b f \cdot \nabla_b f)^{-\frac{1}{2}}= - \frac{1}{| df|^3 }\cdot (\nabla^b f)(\nabla_a\nabla _b f),
$$

we have:
\begin{prop}\label{Hprop}
  Let $f:M\to \mathbb{R}$ be a slice function. Then $H_f:M\to
  \mathbb{R}$, given by
  \begin{equation}\label{Hf-form}
    (n-1) H_f= \frac{1}{| df|} \Delta f- \frac{1}{|df|^3}\cdot (\nabla^a f) (\nabla^b f)(\nabla_a\nabla _b f),
  \end{equation}
  satisfies that, at every $p\in M$, $H_f(p)$ is the mean curvature
  of the smooth hypersurface level set $\{x\in M\mid f(x)=f(p)\}$.
  \end{prop}

We will say a slicing, as in Proposition \ref{Hprop} is a {\em CMC
  slicing} if $H_f$ is constant on each leaf of the slicing (but
different leaves may have different mean curvature).

\begin{example}\label{two}
  In Euclidean space with the origin removed, $\mathbb{R}^n\setminus
  \{0\}$, we have the sphere slicing given by the level sets of
  $f:=r^2=(x^1)^2+\cdots +(x^n)^2$. Then, with our sign conventions,
  $H_f=\frac{1}{r}$.

    \end{example}

By the construction of $H_f$, we have the following Lemma.
\begin{lemma}\label{invariance}
Let $f:M\to \mathbb{R}$ be a slice function. Then $H_f =H_{F\circ f}$ (or $H_f =-H_{F\circ f}$ ) where $F:\mathbb{R}\to \mathbb{R}$ is
any smooth strictly increasing (or, respectively, strictly decreasing) function.
\end{lemma}

\begin{proof}
   The result is immediate from  the first part of  Lemma \ref{slice-lemma}. Or we may use from the proof there that
  $$
  \nu_{F\circ f}=\pm \nu_f  \qquad \mbox{at all }\qquad p\in N,
  $$
  and then that
Formula (\ref{Hf-form}) arises from expanding (\ref{mean-ex}).
  \end{proof}

\subsection{CMC slicings}

Some applications of the above are immediate.
\begin{prop}\label{cmc1}
Let $f:M\to \mathbb{R}$ be a slice function.
The level sets of $f$ give a CMC slicing iff
\begin{equation}\label{main}
 H_f=G\circ f
\end{equation}
for some
smooth function $G:\mathbb{R}\to \mathbb{R}$.
\end{prop}
\begin{proof}
  $\Leftarrow :$   $G\circ f$ is clearly constant on the level sets of $f$. Thus if \eqref{main} holds then those levels sets are CMC.\\

  $\Rightarrow :$ Consider a point $p\in M$ and the (hypersurface-)leaf of
  the slicing that contains $p$.  In a sufficiently small open
  neighbourhood of $p$, that is a local collar of the leaf containing $p$,  we can find coordinates
  $(x^1,\cdots, x^n)$  mapping the neighbourhood diffeomorphically onto the open  open set
  $U\times V\subset \mathbb{R}\times\mathbb{R}^{n-1}$ so that, on this neighbourhood, $f=x^1$.

  Suppose now the slicing is CMC. Then $H_f$ is independent of
  $(x^2,\cdots, x^n)$. So $H_f=G\circ x^1$, for some smooth function
  $G:U \to \R$ where $U\subset \R$. This establishes the claim
  locally.  Considering all points, $G$ must match on overlaps so the
  claimed result holds globally.
\end{proof}

Given a CMC
slicing, the PDE $H_f=G(f)$ does not necessarily determine $f$, even if
$G$ is specified.  For example if $H_f$ is $\operatorname{constant}$ then
$f$ is clearly not determined by (\ref{Hf-form}), and we see the
ambiguity arising as in Lemma \ref{invariance}: if $f$ is a solution
then so is $F\circ f$ for any smooth strictly increasing function $F$.

There are two cases where this problem is easily resolved, the first
of which is the generic setting.
\begin{definition}
  A CMC foliation is
           {\bf generic} if its mean curvature function $H$ satisfies
           that $dH$ is nowhere zero.
\end{definition}
\noindent If the mean
curvature function $H_f$ satisfies that $d H_f$ is nowhere then
we can remove the freedom described in Lemma \ref{invariance} and
normalise our choice of $f$ representing the slicing, as follows.
\begin{customthm}{{\bf 2}}

 \it{A  foliation is  generic CMC iff it admits a  slice function $f$ satisfying \eqref{CMCeq}
\[
(n-1)\lambda f= \frac{1}{| df|} \Delta f- \frac{1}{|df|^3}\cdot (\nabla^a f) (\nabla^b f)(\nabla_a\nabla _b f) ,
\]
where $\lambda$ is either $+1$ or $-1$.}
\end{customthm}

\begin{proof}
  $\Rightarrow$ Since the foliation is CMC, the mean curvature
  function $H$ is constant on the leaves.  Since $dH$ is nowhere zero,
  $H$ is a slice function. Thus the level sets of $H$ are the leaves
  of the foliation and $H$ is a slice function for the given
  foliation.

Thus the mean curvature function $H_H$ for the slicing defined by $H$ (viewed as a slice function) satisfies
  $H_H=\pm H$. That is $f:=H$
solves (\ref{CMCeq}). \\
$\Leftarrow$ Let $f$ be a slice function satisfying (\ref{CMCeq}).
The right hand side of (\ref{CMCeq}) is
  $(n-1)H_f$. Thus if (\ref{CMCeq}) holds then $H_f=\lambda f$, whence
  $d H_f$ is nowhere zero, and slicing is CMC by previous Proposition.
  \end{proof}

So generic CMC foliations are transversely orientable and are given by slice functions $f$ satisfying \eqref{the-eq}

\begin{equation*}
H_f= \lambda f .
\end{equation*}
It is evident from (\ref{CMCeq}) that $\lambda$ is
independent of the sign of $f$, so we have the following.
\begin{cor}\label{invts}
  (1) the sign $\lambda$ is an invariant of unoriented generic CMC slicings,\\
  (2) $f$ is a characterising invariant of transversely-oriented  generic CMC foliations.
  \end{cor}

The next case is the opposite extreme to that just treated. Namely
that the mean curvature function is everywhere constant. For completeness, we observe here that, in this setting, to an extent we can normalise the slice function.
\begin{prop}\label{cmc2}
     Suppose that a  slicing, given by a defining function $\tilde{f}$, has
     $H_{\tilde{f}}=H$ is constant, so that $dH_{\tilde{f}}=0$ at all
     points. Suppose also that there is a smooth regular curve $\gamma$ on $M$ that is
     transverse to every leaf of the foliation.
     Then there is a unique slice function $f: M \to \R$ (up
     to adding a constant) satisfying the equation

    \begin{equation}\label{CMCC}
(n-1) H= \frac{1}{| df|} \Delta f- \frac{1}{|df|^3}\cdot (\nabla^a f) (\nabla^b f)(\nabla_a\nabla _b f) ,
\end{equation}
    and
$$
 |d(f\circ \gamma)|=1
$$
on the domain of $\gamma$ in $\mathbb{R}$.
\end{prop}

\begin{proof}
Choose any slice function $f$ giving the slicing. Then using
Proposition \ref{Hprop}, and by a change of its sign (to $-f$) if
necessary, it then must satisfy equation (\ref{CMCC}). Using Lemma
\ref{invariance} it follows that we can replace $f$ with a function
$F(f)$ where $F'>0$, and so $|d(F\circ f)|=F'(f)|df|$. Fixing a choice
of smooth curve $\gamma$ that meets every leaf transversely. (It is
straightforward to see that such a curve always exists locally.) It is
elementary to see that we can solve $|d(F\circ f\circ \gamma) |=1 $
and then this determines $F:\R\to \R$, up to the addition of a
constant (function).
\end{proof}

Concerning the above theorem, on suitable local sets one can normalise
the choice of $\gamma$ by requiring it to be parametrised by arc
length and orthogonal to the leaves.  Then we can set $f$ to zero on a
nominated leaf. Thus it then depends on only the choice of a point on that
leaf (and the curve through that point arises by integrating the unit
normal field).

In summary any slicing with constant mean curvature function satisfies  the PDE (\ref{CMCC}) for some constant $H$, and then $f$ is determined up to the choice of $F$ as in Lemma \ref{invariance}, but can be normalised to the extent of this Proposition \ref{cmc2}.

\section{CMC hypersurface foliations - some comments} \label{CMCfol}

Here we make some comments concerning slicings that relate also to global
CMC foliations. First note that if a foliation is CMC then its mean
curvature function satisfies
  \begin{equation}\label{CMCeq-rat}
  (n-1)\lambda H|dH|^3 = |dH|^2 \Delta H-
    (\nabla^a H) (\nabla^b H)(\nabla_a\nabla _b H) ,
  \end{equation}
  locally at every point for some $\lambda = 1$, or $-1$.
 To see this we use that locally the foliation is given as a slicing  by some
  slice function $f$. Let us fix and work on an open set on which this holds.
  If $dH_f$ is zero
  at some point then, at that point, (\ref{CMCeq-rat}) is satisfied
  trivially. Otherwise if at some point $p$ we have $dH_f(p)\neq 0$
  then $dH_f$ is nowhere zero in an open set containing $p$ and then, in that open set,
  (\ref{CMCeq-rat}) is satisfied by the $\Rightarrow$ direction of
  part (1) of the Theorem \ref{cmc}.

One might hope to have a converse to the observation just made, that
is, that if (\ref{CMCeq-rat}) holds then we can conclude that the
foliation is CMC. Clearly if a slicing has $dH=0$ then it is CMC.  For
a slice function $f$ if (locally) $dH_f$ is non-vanishing and
satisfies (\ref{CMCeq-rat}) then $H_f$ is a slice function for a CMC
slicing. But we don't know that this is the slicing determined by $f$.
Stating this more formally, the problem is that if a foliation,
represented locally by $f$, has $H_f$ satisfying (\ref{CMCeq-rat})
then we do not know that, where $dH_f\neq 0$, the 1-forms $df$ and
$dH_f$ have the same annihilator.
It seems unlikely that a converse can be
established: We show that infinitesimally $H_f$ does not uniquely
determine $f$ up to
the ambiguity described in Lemma \ref{invariance}, even when $d H_f$
is nowhere vanishing. We see this by linearising the map $f\mapsto H_f$.

\begin{prop}\label{conversenonuniqueness}
The linearisation of the operator $f\mapsto H_f$ can have a greater that 1-dimensional kernel,
even at points $f$ where $dH_f$ is nowhere zero.
\end{prop}

\begin{proof}

We start by linearising the operator $f\mapsto H_f$. Let us consider a
variation around a given $f:M\to \mathbb{R}$ using the variation
function $u:M\to \mathbb{R}$, with parameter $t$ (in some some open
interval containing zero), yielding $f_t:M \to \mathbb{R}$ with
$f_t=f+tu$.

Then one can compute the variation of the length as follows

\begin{align*}
\frac{d}{dt}\frac{1}{|df_t|}\Big|_{t=0}= -\frac{\IP{df}{du}}{|df|^3}.
\end{align*}

So the unit normal field generated by the level set function $f_t$ varies as

\begin{align*}
\frac{d}{dt}\frac{df_t}{|df_t|}\Big|_{t=0} = \frac{du}{|df|} - \frac{\IP{du}{df}}{|df|^3}df.
\end{align*}

While computing the variation through the operator we note that this is an ambient operator so the variation will only affect the normal $\nu_t=\frac{df_t}{|df_t|}$ in \eqref{firstH}.

\begin{align*}
\frac{d}{dt}\Big|_{t=0}H_{f_t}&=\frac{1}{n-1}\frac{d}{dt}\Big|_{t=0}\nabla_a \nu_t^a = \frac{1}{n-1}\nabla_a \frac{d}{dt}\Big|_{t=0}\nu_t^a\\
&= \frac{1}{n-1} \nabla_a \bigg(\frac{\nabla^a u}{|df|} - \frac{\nabla^b u \nabla_b f}{|df|^3}\nabla^a f\bigg).
\end{align*}
We want to consider variations of $f$ by variation functions  that cannot be
expressed as a composition of another function $F:\mathbb{R}\to
\mathbb{R}$ and $f$. For our purposes it will thus suffice to consider maps $
u:M\to\mathbb{R}$ such that $\IP{du}{df}=\nabla_a u\nabla^a f=0$.

Given that the second term inside the divergence then vanishes, the linearisation around $f$ above becomes
\begin{align*}
\frac{d}{dt}\Big|_{t=0}H_{f_t}&= \frac{1}{n-1} \bigg(\frac{\Delta u}{|df|} - \frac{\nabla^a u \nabla^b f \nabla_a\nabla_b f}{|df|^3}\bigg).
\end{align*}
Using orthogonality we can also move a derivative from the $f$ to $u$ in the last term to obtain the following.
\begin{align*}
\frac{d}{dt}\Big|_{t=0}H_{f_t}&= \frac{1}{n-1} \bigg(\frac{\Delta u}{|df|} + \frac{\nabla^a f \nabla^b f \nabla_a\nabla_b u}{|df|^3}\bigg).
\end{align*}

Now we find an example  for which we have
that the linearisation of  $f\to H_f$ has variations as above in its kernel.

Take $M=\mathbb{R}^2\setminus\{ (0,y)\mid y\in \mathbb{R}\}$; that is
two dimensional Euclidean space without the $y$-axis. We foliate this
space by circles (please see Example \ref{E2} below for more details)
using the level sets of the slice function $f:\mathbb{R}^2\setminus\{0\}\to
\mathbb{R}$
 $$
f(x,y)=(x)^2+(y)^2
$$
The derivative  of this function is given by $df=2(x \ y)$.

We take as variation function
$u:\mathbb{R}^2\setminus\{ (0,y)\mid y\in \mathbb{R}\}\to \mathbb{R}$,
which gives an angle parametrisation of the (part) circles foliating $M$,
 $$
u(x,y)=\arctan(\frac{y}{x}),
$$
which has the following gradient $du=(-\frac{y}{x^2+y^2}
\ \frac{x}{x^2+y^2})$. Then  the two derivatives are
orthogonal, $\IP{df}{du}=0$. To check that the infinitesimal variation of $H_f$ in
directions given by $u$ is zero we need to show that
\begin{align}
0 = \frac{d}{dt}\Big|_{t=0}H_{f_t}&= \frac{1}{n-1} \bigg(\frac{\Delta u}{|df|} + \frac{\nabla^a f \nabla^b f \nabla_a\nabla_b u}{|df|^3}\bigg).\label{toshow}
\end{align}
This can be easily verified by computing the Hessian of $u$:
\begin{align*}
\nabla^2 u = \begin{pmatrix}
\frac{2xy}{(x^2+y^2)^2} & \frac{y^2-x^2}{(x^2+y^2)^2}\\
 \frac{y^2-x^2}{(x^2+y^2)^2} & -\frac{2xy}{(x^2+y^2)^2}
\end{pmatrix}
\end{align*}
We first note that the function
$u$ is harmonic, that is $\Delta u=0$, and also that the second term
in \eqref{toshow} will vanish by computing
\begin{align*}
\nabla^a f \nabla^b f \nabla_a\nabla_b u =\frac{8x^3y+8(y^2-x^2)xy-8xy^3}{(x^2+y^2)^2} = 0.
\end{align*}

\end{proof}

\bigskip

Next some  comments in a different direction. While we should think of (\ref{the-eq}) as the canonical way to
describe a generic CMC slicing, there can be some gains from exploiting
Lemma \ref{invariance} to choose an alternative slice function -- in
particular to yield a function that smoothly extends to regions
where the mean curvature and slicing is singular.

Using Lemma \ref{invariance} and Proposition \ref{cmc1} we see
that we have  alternative generic CMC equations given by
(\ref{the-eq})
\begin{equation}\label{var-the-eq}
H_f= \lambda  G\circ f
\end{equation}
where $G:\R\to \R$ is any {\em fixed} smooth strictly monotonic function (because $G$ is fixed we require $\lambda$).

 Of course there is no gain if $df$ is
nowhere zero, as required of slice functions. But there can be if wish to
generalise our setting slightly.

\begin{example}\label{E2}
  In Euclidean space with the origin removed, $\mathbb{R}^n\setminus
  \{0\}$, we have the sphere foliation given by the level sets of
  $\tilde{f}:=\frac{1}{r}$ (which is the $H_f$ of Example \ref{two}) where $r=\sqrt{(x_1)^2+\cdots +(x_n)^2}$. Then $H_{\tilde{f}}=-\frac{1}{r}=-\tilde{f}$. So $\tilde{f}:=\frac{1}{r}$ solves
  (\ref{the-eq}) with $\lambda =-1$.

  But notice that if we take
  $f:=F\circ \tilde{f}=(H_f)^{-2} = (\tilde{f})^{-2}$ (so $F'<0$ on
  $\mathbb{R}_{\geq 0}$) then we come to $f= F\circ \tilde{f}=r^2$, as in Example \ref{two}. Then we have
  $$
H_f= G\circ f
$$
where $G(f):=f^{-\frac{1}{2}}$.

Now multiplying through by $|df|^3$  we come to
  $$
(n-1)|df|^3 f^{-\frac{1}{2}}= | df|^2 \Delta f-  (\nabla^a f) (\nabla^b f)(\nabla_a\nabla _b f) .
  $$
so that now this is satisfied by
  $$
f=(x_1)^2+\cdots +(x_n)^2
$$
on all of $\mathbb{R}^n$. Thus, in this example, there is an
advantage in using $f=r^2= H_{\tilde{f}}^{-2}$, rather than the mean curvature
itself, in that we obtain an equation with solution on $\mathbb{R}^n$
that extends the foliation equation and solution  on $\mathbb{R}^n\setminus \{0\}$ -- although now $f$ is only a slice function on the open dense set
$\mathbb{R}^n\setminus \{0\}$.
  \end{example}

\begin{example}\label{E3}
  Consider the $n$-sphere with its usual metric and standard CMC
  foliation by totally umbilic $(n-1)$-spheres. We will take intially
  $x_{n+1}$ as the slice function. Then the mean curvature $H$ is zero
  along the equator and is infinite at the poles.  In the standard
  $\mathbb{R}^{n+1}$ coordinates the sphere is the set where
  $x_1^2+\cdots +x_{n+1}^2=1$ and we may view $(x_1,\cdots , x_n)$ as
  coordinates on the (for example) $x_{n+1}>0$ part of the sphere
  (sometimes called Monge patch coordinates).  In these coordinates,
  and with a suitable choice of orientation, the mean curvature of a
  $n-1$-dimensional sphere in the $n$-dimensional sphere is given by
  $$\displaystyle
  H=-\frac{x_{n+1}}{\sqrt{1-x_{n+1}^2}}= - \frac{\sqrt{1-\sum_{i=1}^n
      x_i^2}}{\sqrt{\sum_{i=1}^n x_i^2}} .
  $$
  Note that the sign is due the choice of slice function.

    Setting $f:=G(H)=\frac{1}{1+H^2}$ we obtain that
  $f$ extends to be  a smooth function on the (entire) $n$-sphere (namely $f=x_1^2+\cdots +x_n^2$) with
  values in $[0,1]$ and solves
  $$
H_f=\mp H
  $$
away from the pole and the equator on the North and, respectively, Southern hemispheres. Note that $df$ is then zero at the equator ($df$ on sections of the sphere tangent bundle) and the poles.
That is
  $$
(n-1)|df|^3 (\frac{1}{f}-1)^{\frac{1}{2}}= | df|^2 \Delta f-  (\nabla^a f) (\nabla^b f)(\nabla_a\nabla _b f).
$$ holds away from the poles and the left hand side has a removable
singularity at the poles and so in this sense it holds globally.

\newcommand{\tf}{t}
In fact we can also use the height function $\tf=x_{n+1}$ itself, which is a slice
function on the sphere minus its poles (and is a Morse function). There  it is
$$
\tf= -\frac{H}{\sqrt{1+H^2}}.
$$
This satisfies
$$
-(n-1)\tf |d\tf|^3  = \sqrt{1-\tf^2}\left(| d\tf|^2 \Delta \tf-  (\nabla^a \tf) (\nabla^b \tf)(\nabla_a\nabla _b \tf)\right),
$$
away from the poles, and again the left hand side has a removable
singularity at the poles.
\end{example}

\section{Conformal aspects}\label{conf-sec}

Consider a fixed hypersurface $\Sigma$ in a Riemannian manifold $(M,g)$. We now
consider the implications of replacing our Remannian metric $g$ with a
conformally related metric $\widehat{g}=e^{2\om}g$, where $\om\in C^\infty
(M)$. For the mean curvature of $\Sigma$, it is straightforward to compute that for such conformally related metrics
we have
\begin{equation}\label{H-trans}
H^{\widehat{g}}=e^{-\om}(H^g+\nu^a\nabla_a \om ),
\end{equation}
where on the right hand side $\nu^a$ is the normal field, which is unit
with respect to $g$, (and $\nabla w=dw$, the exterior derivative of
$f$, and so is independent of any metric). Thus we have the following
result.

\begin{customprop}{{\bf 3}}
  \it{ Given a slicing $f$ on a Riemannian manifold
    $(M,g)$, there is locally a conformally related metric
    $\widehat{g}=e^{2\om}g$ so that the slicing is minimal with
  respect to $\widehat{g}$.

  This extends to a global result if the manifold is contractible or it is  a
  product compatible with the slicing.}
  \end{customprop}
\begin{proof}
Clearly there is no loss of generality by dealing with the case
that $M$  is connected, so we assume this.

Let $\ell$ denote a leaf of the slicing, and set $t=f+C$ where the constant $C$ is chosen so that $\ell$ is the  zero locus of $t$.

Now we have the vector field $\bar{\nu}_t:=g^{-1}(dt,\cdot)$ and this
determines a flow transverse to the leaves of the foliation. Locally, for an open set $U$ of the leaf, choose
coordinates $(x^1,\cdots, x^{n-1})$
on $\ell$ and extend these to
functions on a cylinder $\mathcal{C}=(-\epsilon,\epsilon) \times U\subset
M$
so that each function $x^i$ is constant along the
flow. Thus $\bar{\nu}_t\cdot x^i =0$, for each $i=1,\cdots ,n-1$.

Because $\bar{\nu}_t\cdot x^i =0$ there is a function $\alpha$ such that $\bar{\nu}=\alpha \frac{\partial}{\partial t}$. We compute
\begin{equation*}
\alpha=\bar{\nu}_t\cdot t = dt (\bar{\nu}_t) =\bar{\nu}^a dt_a = g^{ab} dt_a dt_b=g^{tt}.
\end{equation*}
In particular $\alpha>0$.

Thus in the coordinates $(t,x^1,\cdots, x^{n-1})$ we have
\begin{equation*}
\nu_t =\frac{\bar{\nu}}{|\bar{\nu}|}=\sqrt{g^{tt}}\partial_t.
\end{equation*}

In view of (\ref{H-trans}), to establish the proposition on the cylinder $\mathcal{C}$ it is necessary and sufficient  to solve the equation
$$
H^g+ \sqrt{g^{tt}}\frac{\partial}{\partial t} \om=0 .
$$
But, using the coordinates introduced, this is achieved by a direct integration
\begin{align*}
\om(t,x^1,\cdots , x^{n-1})&:=-\int_0^t\frac{1}{\sqrt{g^{ss}}} H^g(s,x^1,\cdots , x^{n-1}) ds \\
&= -\int_0^tH^g(s,x^1,\cdots , x^{n-1})\sqrt{g_{ss}} ds.
\end{align*}
This establishes the result locally. Note that $\sqrt{g_{ss}}ds$ is the measure induced on the flow lines from the ambient Riemannian metric.

We can extend to include the entire leaf $\ell$ as follows.
Working on any other open set of the leaf $\ell$ we can repeat. On
overlaps of open cylinders $\mathcal{C}_1$ and $\mathcal{C}_2$ the
respective solutions $\om_1$ and $\om_2$ agree up to the addition of a
function that is annihilated by $\nu_t$. But the function is zero as
$\om_1$ and $\om_2$ both vanish on the leaf $\ell$. Thus the problem
is solved in a neighbourhood of the leaf $\ell$.

Similarly for the solution $\om_3$ in a cylinder $\mathcal{C}_3$ based
around a nearby leaf, where say $t=t_0\in (-\epsilon,\epsilon)$. On
the overlap between this and $\mathcal{C}_1$ the solution must differ
from $\om_1$ by a function annihilated by $\nu_t$, and this can made
zero by requiring that, at $t=t_0$, $\om_3$ agrees with
$\om(t_0,x^1,\cdots , x^{n-1})$. These observations allow one to extend to a global result in the two cases mentioned.
\end{proof}

\begin{example}[CMC foliation of the sphere conformally deformed into a minimal foliation]
  Consider the $n$-sphere (minus its poles) with its usual metric
  $g_s$ and standard CMC foliation by totally umbilic $(n-1)$-spheres
  as in previous Example \ref{E3}. In the standard $\mathbb{R}^{n+1}$
  coordinates we again write the ``height'' $x_{n+1}$ as $t$.  In
  terms of $t\in(-1,1)$, the mean curvature of an $n-1$-dimensional
  sphere in the $n$-dimensional sphere is given by
  $$\displaystyle
  H^{g_s}=-\frac{t}{\sqrt{1-t^2}}.
  $$
  To conformally deform the current metric to a metric $\mathring{g}=e^{2\om}g_s$ such that the foliation becomes minimal we must find a conformal factor $\om$ such that

  $$
  0=H^{\mathring{g}}=e^{-\omega}(H^{g_s}+\nu^a\nabla_a \omega)\qquad \Leftrightarrow \qquad   0=H^{g_s}+\nu^a\nabla_a \omega .
  $$

  For the height function $t$ the derivative in the normal direction to the foliation is given by
  $$\nu^a\nabla_a \omega =\frac{1}{\sqrt{1-t^2}}\frac{d \omega}{dt}$$.

  Thus the equation is $\frac{d}{dt}\omega=t$ with e.g.\ the solution
  $\omega(t)=\frac{1}{2}t^2$. Thus with respect to the metric
  $\go=e^{t^2}g_s$  the foliation is
  minimal.
\end{example}

\medskip

 Next we use Proposition \ref{min-prop} to yield a simple approach to   a more general result.
\begin{customthm}{{\bf 4}}
  \it{ Let $f$ be a slice function on a  Riemannian
   manifold $(M,g)$.
    Locally,  we can conformally prescribe the mean curvature to be any smooth
    function.
That is, given a smooth
    function $h$ on $M$, there is a metric
   $\widehat{g}$ in the conformal class $[g]$ such that the equation
  \begin{equation*}
h=H^{\widehat{g}}_f
  \end{equation*}
  is satisfied.

This extends to a global result for smooth functions $h$ with compact
support if the manifold is contractible, or it is a product compatible
with the slicing.}
\end{customthm}

\begin{proof}
Working locally, using Proposition \ref{min-prop} there is a metric
$\mathring{g}$ conformally related to $g$ such that
$H^{\mathring{g}}=0$.  In view of (\ref{H-trans}) we need to solve for
$\omega$ the following equation $h=H^{\widehat{g}}=
e^{-\om}\nu^a\nabla_a \om=- \nu^a\nabla_a e^{-\om} $, where
$\widehat{g}=e^{2\omega}\mathring{g}$.

Set $t=f$. Define coordinates as in the proof of Proposition \eqref{min-prop}. We then have
$$
H^{\widehat{g}}= - \sqrt{\mathring{g}^{tt}} \frac{\partial}{\partial t}e^{-\om}.
$$

Thus, given an arbitrary smooth function $h:M\to \mathbb{R}$,
we want to solve

\begin{equation*}
\frac{\partial }{\partial t}e^{-\om} = - \frac{1}{\sqrt{\mathring{g}^{tt}}} h.
\end{equation*}

 So we want
$$
e^{-\om} - C(x_1,\ldots,x_{n-1}) = -\int_{t_0}^t  \frac{1}{\sqrt{\mathring{g}^{ss}}} h(s,x^1,\cdots , x^{n-1}) ds
$$
where $C(x_1,\ldots,x_{n-1})$ is a function of $(x_1,\ldots, x_{n-1})$ and $t=t_0$ is a leaf of the slicing.
Thus
$$
\om =-\log \big(~C- \int_{t_0}^t  \frac{1}{\sqrt{\mathring{g}^{ss}}} h(s,x^1,\cdots , x^{n-1}) ds~\big),$$
or,
$$
= -\log \big(~C- \int_{t_0}^t  \sqrt{\mathring{g}_{ss}} h(s,x^1,\cdots , x^{n-1}) ds~\big) ,
$$
solves   $H^{\widehat{g}} =h$ for any bounded function $h$ with compact support, since we may choose $C$ so that
 $C- \int^t_{t_0}  \sqrt{g_{ss}}h $ is positive.

This solves the problem locally. For global the statement, given the
assumptions we can access $\go$ globally, from Proposition \ref{min-prop}.
Then arguing as in the Proof of Proposition \ref{min-prop} we can find and
match the choices of $C$ as, given that $h$ has compact support, the subset of $\mathbb{R}$ formed by the collection of
integrals $ \int_{flowline} h $ (over each flow line) is bounded above.
\end{proof}

\noindent From the proof we see that the Theorem holds for a much
larger class of functions $h$ than the set of those  with
compact support. We really just need that $h$ is properly integrable
on each connected flow line of the vector field
$\mathring{g}^{-1}(df,\cdot)$, and the set of such integrals is
bounded above.

\begin{remark} The proof above employs a
similar idea to that used to prove Theorem $1.(ii)$ of
\cite{walczak1984mean}.  In that source the author uses a conformal
rescaling combined with a scaling of a single metric component, so overall a
change of metric that is not purely conformal. The combination enables a stronger result in terms of what functions may be prescribed. Here we focus on what can be attained conformally.
\end{remark}

As a particular application of Theorem \ref{conf-presc}, and its
proof, we see that, given a Riemannian hypersurface slicing, there is
a conformally related metric that makes this CMC, at least
locally. This follows by setting $h$, in (\ref{p-eq}), to be a real-valued function $G$ composed with the slice
function.  We state this formally as follows.
\begin{cor} \label{c-solve}
  Let $f$ be a slicing on a Riemannian manifold $(M,g)$, satisfying the conditions of Proposition \ref{min-prop}, and $G$ any smooth function $G:\mathbb{R}\to \mathbb{R}$. Suppose that
       $\int G \circ f$ is integrable on each connected flow line of the vector field $\mathring{g}^{-1}(df,\cdot)$, and the set of such integrals is bounded above, then
    there is a metric $\widehat{g}\in [g]$ such that
   (\ref{main}) is satisfied, that is
 $$
H^{\widehat{g}}_f =G\circ f.
$$
So the slicing determined by $f$ is CMC for the metric
$\widehat{g}\in [g]$.

 In particular given $\lambda:=\pm 1$ and specialising to   $G(f):=\lambda f$ (with the same assumptions on the flow line integrals)
 there is $\widehat{g}$, conformally related to $g$, such that
  $$
H^{\widehat{g}}_f =\lambda f .
$$
So the slicing determined by $f$ is generic CMC for the metric
$\widehat{g}\in [g]$ and $H_f$ solves (\ref{CMCeq}).
  \end{cor}

\begin{remark}
  Note that this result proves that for the  $\lambda$ defined in
  Theorem \ref{cmc} (and cf. Corollary \ref{invts}) both signs may arise.
  See Example \ref{mintoCMC} for an explicit example. Once again the assumptions of the theorem could be reduced if we allow also non-conformal changes of metric, as in e.g.\ \cite{walczak1984mean}.
\end{remark}

\begin{example}[Euclidean to Hyperbolic space - Half plane model]
  \label{Eucl-half}
  Consider the upper half space in Euclidean $\R^3$ foliated with minimal planes
  $\{z=c\}$ for $c\in[0,\infty)$.  The metric of the hyperbolic
    half-space model is
    $g_h=\frac{dx^2+dy^2+dz^2}{z^2}$ where the Euclidean metric is
    $g_E=dx^2+dy^2+dz^2$. The conformal factor  then is
    $e^{2\om}=\frac{1}{z^2}$ giving $\om(z)=-\log z$.

Using this conformal factor the foliation of the Euclidean space by minimal surfaces is transformed into a non-minimal CMC one for the hyperbolic half space:
$$
  H^{g_H}=e^{-\omega}(H^{g_E}+\nu^a\nabla_a \omega)=e^{-\omega}\nu^a\nabla_a \omega = z \cdot \frac{d}{dz}(-\log z)=-1.
  $$

\end{example}

\begin{example}[Euclidean to Hyperbolic space - Poincar\'e disk model]\label{Poin-disk}
Following from the previous example we consider the Poincar\'e Disk
model of the hyperbolic space. The metric is here
$g_H=\frac{4g_E}{(1-r^2)^2}$ where, again, $g_E=dx^2+dy^2+dz^2$ denotes the
Euclidean metric. Here $r^2=x^2+y^2+z^2$.

The conformal factor between these two metrics is
$\om(r)=\log\big(\frac{4}{(1-r^2)^2}\big)^{\frac{1}{2}}$. The CMC
foliation of the Euclidean half space changes to CMC one of the
Poincar\'e Disk as
\begin{align*}
H^{g_H}&=e^{-\omega}(H^{g_E}+\nu^a\nabla_a \omega)=e^{-\omega}(\frac{1}{r}+\nu^a\nabla_a \omega)\\
&= -\frac{1-r^2}{2} \Big(\frac{1}{r}+ \frac{d}{dr}\bigg(\log\big(\frac{4}{(1-r^2)^2}\big)^{\frac{1}{2}}\bigg)\Big)=-\frac{r^2+1}{2r}.
\end{align*}

\end{example}

\section{Conformally compact manifolds} \label{ccpct}

Let $\ol{M}$ be a compact $n$-manifold with boundary, and write $M$
for the interior. So $\ol{M}=M\cup \partial M$ where the boundary
$\partial{M}$ is a smooth closed $(n-1)$-manifold. Recall that a metric $\gp$ on
$M$ is said to be conformally compact if the following holds:
there is a metric $g$ on $\ol{M}$ such that, in some collar
neighbourhood of $\partial M$, we have $g=r^2 \gp$ with $r$ a slice
function such that its zero locus is exactly $\partial M$, i.e.,
$$
\partial M = \mathcal{Z}(r).
$$

Recall also that a slice function $r$ is said to be defining for
the boundary $\partial M$ if this last property holds, that is
$\partial M=r^{-1}(0)$ and we say that we say that $\gp$ is {
  asymptotically hyperbolic (AH)} if $|dr|_g=1$ {\em along} $\partial
M$.

For AH manifolds there is a useful slice function $r$, defined near
$\partial M$, and that is defining for $\partial M$, due to Graham-Lee
\cite{graham1991einstein}. Using this and after a geometric
identification of the collar neighbourhood with a product $\partial
M\times [0,\epsilon)$ (for some $\epsilon >0$), the conformally
  compact metric takes the form
\begin{equation}\label{GL-form}
\gp= \frac{g}{r^2}=\frac{h_r+ dr^2}{r^2} ,
\end{equation}
where $h_r$ is a 1-parameter family of metrics on $\partial M$. In this case $|dr|^2_g=1$, in the given neighbourhood.

As mentioned in the introduction, another useful class of boundary defining
functions would be provided by a positive answer to the following
question:\\
\noindent {\bf Question:} Given $\gp$, is there a defining slice
function $\rho$ so that
$$
g:=\rho^2\gp
$$
is a metric to the boundary and the level sets of $\rho$ are CMC for $g$?

Proposition \ref{step} shows there are many conformally compact
metrics that do admit a positive answer: for any conformal manifold with
boundary there are conformally compact metrics $\gp$ on the interior and
defining function metric pairs $(g,\rho)$ so that $\rho^2 \gp=g$ and
$$
H^g_\rho=\rho.
$$

\bigskip

\begin{proof}[Proof of Proposition \ref{step}]
Let $\tilde{g}\in \cc$.  Take a slice function $\rho$ that defines
$\partial M$. Since the
boundary is compact there is an open collar neighbourhood of $\partial
M$ on which $H_\rho^{\tilde{g}}$ is bounded, and thus also an open collar
neighbourhood on which the conditions of Corollary \ref{c-solve} are
satisfied for the case that $G$ is the identity function.
Now we use Corollary
\ref{c-solve} to conformally transform $\tilde{g}$ to $g=e^{2\om}
\tilde{g}$ so that, with respect to $g$, and on this collar, the slicing has mean curvature
$H_\rho^g=\rho$. We smoothly extend this conformal factor to $M$. Finally we define
$$
\gp:=\frac{g}{\rho^2}
$$
on the interior $M$ (near $\partial M$).
\end{proof}

Given a metric $g$ and CMC slicing, with slice function $\rho$, it is not
necessarily the case that $|d\rho|_g $ is constant along the zero locus
of $\rho$. This follows at once from the coordinate dependence of $C$
in the proof of Theorem \ref{conf-presc}, which may be used to
construct examples.

The proof of Proposition \ref{step} thus constructs conformally
compact metrics $\gp= \frac{g}{\rho^2}$, with a slicing $\rho$ that is CMC
for $g$, which are not in general AH; the examples so constructed are
asymptotically hyperbolic if and only if $|d\rho|_g$ is constant along
the $\rho =0$ leaf. By convention it is often required that
$|d\rho|_g=1$ along the boundary (which is the zero locus of
$\rho$). This means that asymptotically the sectional curvatures of
$\gp$ approach $-1$.

In constrast we note the following.
\begin{prop}\label{cfMPprop}
If $(M,\gp)$ is
conformally compact and $\rho$ is a smooth defining slice function
that is CMC for $\gp$, then  $\gp$ is asymptotically hyperbolic.
  \end{prop}
\begin{proof}
  Since $\rho$ is defining for the conformally compact structure this means that
  $$
\gp =\frac{g}{\rho^2}
$$
where $g$ is a smooth metric to the the boundary $\partial M$ of $M$ in $\ol{M}$.
Using product coordinates so that  $\partial_\rho$ is orthogonal to the boundary (cf.\ the proof of Proposition \ref{min-prop}),  we have $\nu=\sqrt{g^{\rho\rho}}\partial_\rho$. Thus using equation \eqref{H-trans} we compute
\begin{equation}\label{relate}
H^{\gp}_\rho=\rho H^g_\rho-|d \rho|_g .
\end{equation}
This shows that, although $g^+$ is singular at $\partial M$,
$H^{\gp}_\rho$ extends smoothly to the zero locus of $\rho$, as $H^g_
\rho$ is smooth to the boundary. So  if $\rho$ is CMC for $\gp$ then
$H^{\gp}_\rho =F\circ \rho$ for some smooth real valued function $F$ of one variable. Thus $H^{\gp}_0: =\lim_{\rho\to 0} H^{\gp}_\rho=F(0)$ is
locally constant along the boundary. But $ H^{\gp}_0= -|d \rho|_g|_{\partial
  M}$, and so $\gp$ is necessarily asymptotically hyperbolic. (As $\rho$
is defining, $|d \rho|_g$ is nowhere 0 on $\partial M$).
  \end{proof}

\begin{remark}\label{cfMP}
  Much more can be read off from equation \eqref{relate}. First we
  observed above that $H^{g^+}_\rho$ extends smoothly to the
  boundary. It is well known that also the scalar curvature ${\rm
    Sc}^{g^+}$, of a conformally compact metric $g^+$, extends smoothly to the boundary, and its
  limit there satisfies ${\rm Sc}^{g^+}_0=-n(n-1)|d\rho|^2_g$, see
  e.g.\ \cite{mazzeo1988hodge,gover2010almost}. Thus from equation \eqref{relate} we have
  $$
H^{g^+}_0= -\sqrt{- \frac{{\rm Sc}^{g^+}_0 }{n(n-1)} },
  $$
on conformally compact manifolds.

 Next it
is evident from equation \eqref{relate} that if the slicing by $\rho$ is CMC for $g$ then it is
also CMC for $\gp$, and vice versa, if and only if $|d\rho|_g$ is
constant along all the leaves of $\rho$ (in other words $|d\rho|_g=
F\circ \rho$ for some real valued function of one variable $F$).

In another direction if $r^2\gp=g$, where $r$ satisfies
\eqref{GL-form}, then $|dr^2|_g=1$ and it is well known that $H^g_{r=0}=0$.
Thus it follows from \eqref{relate} that
$H^{\gp}_r$ is asymptotically CMC in that
$$
H^{\gp}_r = 1+O(r^2).
$$

Proposition \ref{cfMPprop} and the discussion here show that the
problems we study in this section are not simply related to that of
\cite{mazzeo2011constant}.
\end{remark}

\noindent Proposition \ref{cfMPprop} and the observations made in the Remark
 are also easily shown
using the tractor calculus for conformally compact manifolds as
developed in
\cite{gover2010almost,arias2021conformal,curry2023conformal}. It seems likely that these tools can also provide further geometric insights into  the results of
\cite{mazzeo2011constant}.

\begin{remark}
More generally we can use Theorem \ref{conf-presc} to obtain $g\in
\cc$ and $\gp=\frac{g}{\rho^2}$ so that $H^g$ is almost any desired
function. Or we could use Proposition \ref{min-prop} to achieve that
$H^g_{\rho}=0$ for the $\rho$-slicing.

\end{remark}

\noindent The proof of Proposition \ref{step} depends on a particular choice of
slice function, which then determines $\gp\in \cc|_{M}$.  Thus it does
not lead to an answer to the Problem \ref{question}.

\bigskip

\bigskip

An ideal problem to solve would be the following.
\begin{problem}\label{bigburrito}
  Let $G:[0,\infty) \to \mathbb{R}$ be any smooth function.  Consider a compact $n$-manifold with boundary
    $\ol{M}$ and a conformally compact metric $\gp$ on the interior
    $M$. There is a slice function $\rho$ for the boundary $\partial
    M$ such that, with respect to $g =\rho^2 \gp $ and in a collar neighbourhood of $\partial M$, the foliation
    determined by $\rho$ is CMC for $g$, with $H^{g}_\rho=G\circ \rho$. In particular there is $g$ on $\ol{M}$ in the conformal class so that
  $$
\rho^2 \gp=g \qquad \mbox{and} \qquad H_\rho^g=\rho.
  $$
\end{problem}

A simple setting where we can solve a case of Problem \ref{bigburrito}
(and indeed the variant where we ask $H_\rho^g=-\rho$) is for the
upper half space realisation of the Hyperbolic metric. For simplicity
 we work in dimension 3.

\begin{example}\label{mintoCMC}
Let $\gp=g_H$ denote the hyperbolic metric in the upper half space $\mathcal{H}^3$ of $\R^3$. Then we have $\gp=\frac{g_E}{z^2}$ where $g_E$ is the Euclidean metric. Also $H_z^{g_E}=0$, i.e. the $z=c$ slices are minimal with respect to the Euclidean metric $g_E$ that goes to the boundary.

Now we seek a defining slice function $\rho\in \mathcal{C}^\infty$
such that, at least in a collar neighbourhood of the boundary, we have
$H_\rho^g=\rho$ where $g=\rho^2 \gp$.

This means we want $\om\in \mathcal{C}^\infty(\overline{\mathcal{H}}^3)$ solving

\begin{equation*}
\rho:=e^{\omega} \cdot z=H^g_\rho= e^{-\om} \nu^a \nabla_a \om,
\end{equation*}
where as usual $\nu$ is the unit normal to the slicing with respect to the Euclidean metric.

By the symmetry of Hyperbolic space and its compactification in the upper half space model we expect that $\om$ and hence $\rho$ can be taken to be functions of $z$ alone and independent of the other coordinates. In this case $\nu$ will necessarily be proportional to $\partial_z$ and in fact $\frac{\partial}{\partial z}$ is an inner pointing unit normal thus the problem boils down to solving

$$
e^\om z =e^{-\om} \frac{d}{dz}\om,
$$

where $\om=\om(z)$.

Hence $\om(z)=-\frac{1}{2}\log (c-z^2)$ for a suitable constant $c>0$, which gives $\rho(z)=\frac{z}{\sqrt{c-z^2}}$ and $g=\frac{1}{c-z^2}g_E$.

This solves $H^g_\rho=\rho$, so \eqref{the-eq} with $\lambda=1$, on
$\ol{\mathcal{H}}^3$ with $\gp=\frac{g}{\rho^2}$ on $\mathcal{H}^3$.

Similarly if we set $\om=-\frac{1}{c}\log (z^2+c)$ for some constant, so that
$$
\tilde{g}=\frac{1}{z^2+c}g_E
$$
then we get
$$
H^{\tilde{g}}_\rho=-\rho ,
$$
where $z^2+c$ is nowhere zero, and so solves \eqref{the-eq} now with $\lambda =-1$.
\end{example}

\medskip

In the following for simplicity we restrict to the case of AH metrics $\gp$.
Taking $G$ to be the zero function gives a special case of Problem
\ref{bigburrito}, namely that of foliation by minimal
hypersurfaces. For clarity we state that separately and approach it first.
\newcommand{\br}{\bar{r}}
\begin{problem}\label{minimal}
 Consider a compact $n$-manifold with boundary
    $\ol{M}$ and an AH metric $\gp$ on the interior
    $M$. There is a defining slice function $\br$ for the boundary $\partial
    M$ such that, with respect to $g =\br^2 \gp $ and  in a collar neighbourhood of $\partial M$, the foliation
    determined by $\br$ is minimal, that is
    $$
    H^{g}_{\br}=0 .
    $$
      \end{problem}

Before we attack this,  let us set up some formulas and notation. First, given
a fixed AH metric $\gp$, we have that $g$ is determined by $r$ via $g=r^2\gp$, thus we often drop it in the notation for the mean curvature:
$$
H_r:=H^{g}_r.
$$
~\\

Our strategy will be as follows. Let $(M,\gp)$ be the AH
metric. Choose a metric $g_{\partial M}$ from the conformal class, and
$r$ a corresponding slice function  so that
$$
g=r^2 \gp
$$ is a metric to the boundary $\partial M=\mathcal{Z}(r)$, with $g_{\partial M}$ the boundary metric induced by $g$.  We attempt to construct a formal solution on the boundary for a new slice function
$$
\bar{r}=e^\omega r
$$
satisfying
\begin{equation}\label{eqn}
H_{\bar{r}}= 0,
\end{equation}
where  $H_{\bar{r}}=H_{\bar{r}}^{\bar{g}}$ with $\bar{g}= \bar{r}^2 \gp$. Note that we have $\bar{g}=e^{2\omega} g$.

~\\ Recall $ (n-1)
H^{g}_r=g^{ab}\nabla^g_a \nu_b , $ where $\nu= dr/|dr|_g$.  With $g$
and $r$ fixed, we need to write (\ref{eqn}) as a PDE on $\omega$. We
have
$$
(n-1) H^{\bar{g}}_{\bar{r}}=\bar{g}^{ab}\nabla^{\bar{g}}_a \bar{\nu}_b ,
  $$
  where
  $$
\bar{\nu}_b = \frac{\bar{r}_b}{\sqrt{\bar{g}^{cd} \bar{r}_c\bar{r}_d}} \qquad \mbox \qquad \bar{r}_b:=(d\bar{r})_b
~.   $$
So (recall (\ref{Hf-form}) from Proposition \ref{Hprop})
 \begin{equation}\label{Hr-form}
    (n-1) H_{\bar{r}}^{\bar{g}}= \frac{1}{| d\bar{r}|_{\bar{g}}} \Delta^{\bar{g}} \bar{r}- \frac{1}{|d \bar{r}|_{\bar{g}}^3}\cdot \bar{r}^a \bar{r}^b(\nabla^{\bar{g}}_a \bar{r}_b).
 \end{equation}

Thus we want to find $\om\in C^\infty (M)$ so that
\begin{equation}\label{min-ver}
 (n-1)| d\bar{r}|^3_{\bar{g}} H_{\bar{r}}=| d\bar{r}|^2_{\bar{g}} \Delta^{\bar{g}} \bar{r}- \bar{r}^a \bar{r}^b(\nabla^{\bar{g}}_a \bar{r}_b),
\end{equation}
vanishes (at least formally).

Now $\bar{r}=e^\omega r$ means that
$$
d \bar{r} = e^\omega d r + e^\omega r d\omega = \bar{r}d \omega +e^\omega d r .
$$
Or
$$
\bar{r}_b = \bar{r} \omega_b+ e^\omega r_b  = e^\om r \omega_b+ e^\omega r_b  ,
$$
whence
$$
\bar{g}^{cd}\bar{r}_c\bar{r}_d = e^{-2\om} g^{cd}( e^\om r \omega_c+ e^\omega r_c)( e^\om r \omega_d+ e^\omega r_d) = g^{cd}(  r \omega_c+  r_c)(  r \omega_d+  r_d)  .
$$
So
   $$
\bar{g}^{cd}\bar{r}_c\bar{r}_d = g^{cd} r_cr_d + 2 r g^{cd} r_c\omega_d + r^2g^{cd}\om_c \om_d .
$$

Next we need $\nabla^{\bar{g}}_a \bar{r}_b$. We have
$$
\nabla^{\bar{g}}_a \bar{r}_b = \nabla^{\bar{g}}_a ( e^\om( r \om_b +  r_b ) )= e^\om(\om_a(r \om_b +  r_b) +r_a \om_b +r\nabla^{\bar{g}}_a  \om_b + \nabla^{\bar{g}}_a r_b ).
$$
Now
$$
\nabla^{\bar{g}}_a  r_b = \nabla^g_a r_b -\om_a r_b - \om_b r_a +g_{ab} r_c\om^c
\quad\mbox{and}\quad
\nabla^{\bar{g}}_a  \om_b = \nabla^g_a \om_b -\om_a \om_b - \om_b \om_a +g_{ab}\om_c\om^c.
$$
So
$$
\nabla^{\bar{g}}_a \bar{r}_b = e^\om(\om_a(r \om_b + r_b) +r_a
\om_b + r(\nabla^g_a \om_b -2\om_a \om_b
+g_{ab}\om_c\om^c) + \nabla^g_a r_b -\om_a r_b - \om_b r_a +g_{ab}
r_c\om^c).
$$
This simplifies to
$$
\nabla^{\bar{g}}_a \bar{r}_b  = e^{\om}(\nabla^g_a r_b +r \nabla^g_a \om_b +  g_{ab} r_c\om^c -r\om_a\om_b + r g_{ab}\om_c\om^c ).
$$
 Contracting with $\ol{g}^{ab}$, we obtain
$$
\Delta^{\bar{g}} \bar{r} =e^{-\om}(\Delta^g r +r \Delta^g \om + n r^c\om_c +(n-1)r\om_c\om^c ).
$$

Now our expression (\ref{min-ver}) is
\begin{equation}\label{exp-min-ver}
(n-1)| d\bar{r}|^3_{\bar{g}} H_{\bar{r}}=\bar{r}^a \bar{r}^b (\bar{g}_{ab}  \Delta^{\bar{g}} \bar{r}- \nabla^{\bar{g}}_a \bar{r}_b).
\end{equation}
Substituting all above computed quantities gives us
\begin{align*}
 (n-1)| d\bar{r}|^3_{\bar{g}} H_{\bar{r}}^{\bar{g}}e^{\omega} =(r \omega^a+  r^a)(r \omega^b+  r^b) & [ g_{ab}(\Delta^g r +r \Delta^g \om + n r^c\om_c +(n-1)r\om_c\om^c ) \\
 & - (\nabla^g_a r_b +r \nabla^g_a \om_b +  g_{ab} r_c\om^c -r\om_a\om_b + r g_{ab}\om_c\om^c ) ]
\end{align*}

We can expand the above into
\begin{align}
\label{exp-min-ver-expanded}
(n-1)|d\bar{r}|^3_{\bar{g}} H_{\bar{r}}e^{\omega}&=(n-1)|dr|^3H_r+(n-1)|dr|^2 r^c\omega_c\\
&+r\big(|dr|^2\Delta \omega- r^ar^b\nabla_a\nabla_b \omega+2\omega^ar^b(g_{ab}\Delta r-\nabla_a\nabla_br)\big)\notag\\
&+(2n-2)r\IP{\nabla r}{\nabla \omega}^2 + (n-2)r|\nabla r|^2|\nabla \omega|^2\notag\\
&+r^2F(*^2\nabla \omega, \nabla \omega * \nabla^2 \omega, *^3\nabla \omega ,*^4\nabla \omega)\notag\\
&+ r^3 F(*^4\nabla \omega, *^2\nabla\omega * \nabla^2 \omega)\notag.
\end{align}
Above for the $r^2$ and $r^3$ coefficients we have only recorded the
powers of $\omega$ and its derivatives that appear in the
expansion. In each case $F$ denotes a polynomial (and hence
smooth) function on its arguments. The informal notation $*$ indicates
a part of a tensor product and, for example, we have used the notation
$*^k V$ to denote some tensor part of $k$-fold tensor product of $V$.

We will now construct a formal solution to the
$H_{\bar{r}}^{\bar{g}}=0$ problem, as in Problem \eqref{minimal},
  by an inductive argument.

\begin{proof}[Proof of Proposition \ref{inductionminimal}]

All calculations are in a sufficiently small neighbourhood of the boundary that may change from step to step.
Let $\gp$ be the fixed AH metric considered as in the
Proposition statement. Let $g$ be a compactifying metric. That is $g$
is a metric up to the boundary, $r^2\gp=g$ on the
interior. The AH condition means that the boundary defining slice
function $r$ satisfies
\begin{equation}\label{ah}
|dr|^2_g=1+O(r).
\end{equation}

Set $r_0=r$ and $g_0=g$. We seek $r_1=e^{\omega}r_0$ for some function
$\omega$ smooth up to the boundary, such that $r_1^2\gp=g_1$
and that also $H^{g_1}_{r_1}=r_1F_1=\mathcal{O}(r_1)$ for some function
$F_1$ on $\ol{M}$ that is smooth up to the boundary. Note that, for any such $\om$, $r_1$ is again a non-negative defining function for the boundary.

We posit $\omega: =-r_0H_{r_0}=-rH_r$,

as then
\begin{align*}
r^c\nabla_c(-rH_r)=-|d r|_g^2 H_r-rr^c \nabla_c H_r= -H_r+\mathcal{O}(r)
\end{align*}
where all the above computations are done in the metric $g=g_0$, and
we have used equation (\ref{ah}).  So we have
\begin{align*}
r^c\nabla_c \omega=-H_r+\mathcal{O}(r).
\end{align*}

Using this choice of $\omega$  in equation
\eqref{exp-min-ver-expanded} (for the induction step at hand that is
$\bar{g}=g_1$ and $g=g_0$) we  verify that, on the right hand side, the
first two terms cancel up to  $\mathcal{O}(r_0)$ and the
other terms are manifestly  $O(r_0)$.

Thus we have
\begin{align*}
|d r_1|^3_{g_1} H_{r_1}e^{\omega}= \mathcal{O}(r_0)=r_0\bar{F}_1,
\end{align*}
for some function $\bar{F}_1$ that is smooth up to the boundary.

Now we use that, with $\om$ as here, $e^\om=1+O(r_0)=e^{-2\om}$ (from Taylor's Theorem), so
$r_1=e^{\omega}r_0=r_0+O(r_0^2)$, and we retain
asymptotic hyperbolicity of the $g_1$ metric, in that $|d
r_1|_{g_1}^2=1+\mathcal{O}(r_1)$ at the boundary (see \ref{needed} below).
So
$$
| d r_1|^3_{g_1} H_{r_1}e^{\omega}= \mathcal{O}(r_0)=r_0\bar{F}_1=r_1e^{-\omega}\bar{F}_1,$$
whence
$| d r_1|^3_{g_1} H_{r_1}= r_1 e^{-2\omega}\bar{F}_1=  r_1 \tilde{F}_1$, where $\tilde{F}_1$ is smooth to the boundary, and this gives
$(1+\mathcal{O}(r_1)) H_{r_1}= r_1 \tilde{F}_1$, so finally
$$
H_{r_1}=r_1 F_1.
$$
as required.

For the $k$th induction step, assume that we have
$H_{r_k}=\mathcal{O}(r_k^k)=r_k^kF_k$ for some smooth function $F_k$
up to the boundary. We seek an $\omega_k$, that we will call $\omega$
here (abusing notation) to simplify the exposition, such that with
$r_{k+1}=e^\omega r_k$, and hence $g_{k+1}=e^{2\omega}g_k$,
we have
$H_{r_{k+1}}=\mathcal{O}(r_{k+1}^{k+1})=r_{k+1}^{k+1}F_{k+1}$, for some
smooth function $F_{k+1}$ up to the boundary.

 Generalising the ansatz for $k=0$, we posit
$\omega=-\frac{r_kH_{r_k}}{k+1}=-\frac{r_{k}^{k+1}F_k}{k+1}$ so that
 it solves

\begin{align}\label{cancel}
r_k^c\nabla_c \omega=-H_{r_k}+\mathcal{O}(r_{k}^{k+1}).
\end{align}
Note that this determines $F_k$ up to $+O(r_k)$, and hence we come to the uniqueness statement of the Proposition.
Let us check that our choice of $\omega$ indeed satisfies the above.
\begin{align*}
  r_k^c\nabla_c\big(-\frac{r_k^{k+1}F_{k}}{k+1}\big)
  =-|d r_k|_{g_k}^2 r^k_k F_k+\mathcal{O}(r_k^{k+1}) = -H_{r_k}+\mathcal{O}(r_k^{k+1})
\end{align*}
where we have used  that $|d r_k|_{g_k}^2=1+\mathcal{O}(r_k)$.

Let us now check that if the $g_k$ is asymptotic hyperbolic with the
conformal factor $r_k$ then also the $g_{k+1}$ has the same property
for the choice of $r_{k+1}=e^{\omega}r_k$.
We have
\begin{align*}
  d r_{k+1}&= d\big(e^{\omega}r_k\big)= e^\omega r_kd\omega +e^\omega d r_k,
\end{align*}
So
\begin{align*}
|dr_{k+1}|_{g_{k+1}}^2&=| e^\omega r_kd\omega+e^\omega d r_k|_{g_{k+1}}^2=|d\omega e^\omega r_k+e^\omega d r_k|_{g_k}^2 e^{-2\omega}\\
&=|  r_kd\omega+ d r_k|_{g_k}^2\\
&=|d\omega|_{g_k}^2  r^2_k+2(r^c\omega_c)_{g_k}  r_k + |d r_k|_{g_k}^2\\
&=\mathcal{O}(r_k^{2k+2})+\mathcal{O}(r_k^{k+1})+1+\mathcal{O}(r_k)\\
&=1+\mathcal{O}(r_k)\\
\end{align*}
So we have
\begin{equation}\label{needed}
|dr_{k+1}|_{g_{k+1}}^2= 1+\mathcal{O}(r_{k+1}).
  \end{equation}

With this choice of $\omega$ we return to our equation
\eqref{exp-min-ver-expanded} (for the current induction step -- so
$\bar{g}=g_{k+1}$ and $g=g_k$) to verify that on the right hand side
everything cancels or vanishes mod $+\mathcal{O}(r_k^{k+1})$.

We will look all the terms separately since some of them require some
manipulation.
First, from (\ref{cancel}) above we have
$$
(n-1)|dr_k|_{g_k}^3H_{r_k}+(n-1)|dr_k|_{g_k}^2 r_k^c\omega_c =\mathcal{O}(r^{k+1}_k),
$$
as the $|dr_k|_{g_k}^3$ and $|dr_k|_{g_k}^2$ are each $1+O(r_k)$,
and $H_{r_k}$ and $r_k^c\omega_c$ are each $O(r^k_k)$.
The term
$$
2r_k\omega^ar_k^b((g_k)_{ab}\Delta r_k-\nabla_a\nabla_br_k) =\mathcal{O}(r^{k+1}_k),
$$
is clear as $r_k\omega^a=O(r_k^{k+1})$.
Similarly we have,
\begin{align*}
(2n-2)r_k\IP{\nabla r_k}{\nabla \omega}_{g_k}^2 + (n-2)r|\nabla r_k|^2|\nabla \omega|^2&=\mathcal{O}(r^{2k+1}_k)=\mathcal{O}(r^{k+1}_k),\\
r_k^2F(*^2\nabla \omega, \nabla \omega * \nabla^2 \omega, *^3\nabla \omega ,*^4\nabla \omega)&=\mathcal{O}(r^{2k+1}_k)=\mathcal{O}(r^{k+1}_k),\\
r_k^3 F(*^4\nabla \omega, *^2\nabla\omega * \nabla^2 \omega)&=\mathcal{O}(r^{4k}_k)=\mathcal{O}(r^{k+1}_k).
\end{align*}
The remaining term is $r_k\big(|dr_k|_{g_k}^2\Delta \omega-
r_k^ar_k^b\nabla_a\nabla_b \omega\big)$. At first glance this term
looks like it could generate a problem, as it involves two derivatives
and only one $r_k$ multiplication. But the second order $r$
derivatives cancel out, leaving us with another
$\mathcal{O}(r^{k+1}_k)$ term:
\begin{align*}
r_k\big(|dr_k|_{g_k}^2\Delta \omega- r_k^ar_k^b\nabla_a\nabla_b \omega\big) &= r_k\big(g^{ab}_k-r_k^ar_k^b)\nabla_a \nabla_b\omega + \mathcal{O}(r^{k+1}_k)\\
&= r_k(k+1)k\big(g^{ab}_k-r_k^ar_k^b)(r_k)_a(r_k)_b \frac{(-r_k^{k-1})F_k}{k+1}+\mathcal{O}(r^{k+1}_k)\\
&= -kr_{k}^k|d r_k|_{g_k}^2\big(1- |d r_k|_{g_k}^2 \big)F_k + \mathcal{O}(r^{k+1}_k)\\
&=\mathcal{O}(r^{k+1}_k)
\end{align*}
where we have used (\ref{needed}) (or rather its $k$ version) in the first and last  equalities.

So we can  conclude from the above computations and
\eqref{exp-min-ver-expanded} that
\begin{align*}
  (n-1)| d r_{k+1}|_{g_{k+1}}^3 H_{r_{k+1}}e^{\omega}&= \mathcal{O}(r^{k+1}_k) = \mathcal{O}(r_{k+1}^{k+1})\end{align*}
whence
\begin{align*}
\big(1+\mathcal{O}(r_{k+1})\big)H_{r_{k+1}}&=\mathcal{O}(r_{k+1}^{k+1}),
\end{align*}
where we have used the relation $r_{k+1}=e^\omega r_k$, and again
(\ref{needed}). It follows that
$$
H_{r_{k+1}}=\mathcal{O}(r_{k+1}^{k+1})
$$
which completes the induction.

\end{proof}

We are now ready to treat, formally, the other obvious special case of Problem \ref{bigburrito}, namely when $G$ is the identity function -- so that the mean curvature of the slice function is itself the slice function.

\begin{proof}[Proof of Proposition \ref{inductionCMC}]
Let $\gp$ be the fixed AH metric considered in the statement of
the Proposition. Let $r$ be a non-negative slice function that defines the
boundary and $g$ the corresponding  compactifying metric up to
the boundary. Then
$\gp=\frac{g}{r^2}$, and $|dr|^2_g=1+O(r)$ since $\gp$ is AH.

As in the proof of Proposition \ref{inductionminimal}, set $r_0=r$ and
$g_0=g$. We seek a  function $\omega$,
smooth up to the boundary, so that with $r_1:=e^{\omega}r_0$   the corresponding metric $g_1$ on $\ol{M}$,
satisfying $r_1^2\gp=g_1$ (on $M$), gives
$$
H_{r_1}=r_1+\mathcal{O}(r_1).
$$
For this $k=0$ power of $r$ the
argument is identical to the case of the Proposition
\ref{inductionminimal}: Let $\omega=-r_0H_{r_0}$. Then, as there, we
obtain $H_{r_1}=\mathcal{O}(r_1)$. For our current purposes we rephrase this as
$H_{r_1}=r_1+\mathcal{O}(r_1)$.

The proof of the general inductive step requires more detail than was
needed in the proof of Proposition \ref{inductionminimal}.  For the $k$ to $k+1$
step, we assume that we have
$H_{r_k}=r_k+\mathcal{O}(r_k^k)=r_k+r_k^kF_k$ for some function $F_k$,
that is smooth up to the boundary. We now seek $\omega_k$ (that, as usual,
will denote $\omega$ to simplify the notation) such
that with  $r_{k+1}=e^\omega r_k$, and hence $g_{k+1}=e^{2\omega}g_k$,
we have
$H_{r_{k+1}}=r_{k+1}+\mathcal{O}(r_{k+1}^{k+1})=r_{k+1}+r_{k+1}^{k+1}F_{k+1}$
for some function $F_{k+1}$ that is smooth up to the boundary.

We posit $\omega=-\frac{r_{k}^{k+1}F_k}{k+1}$, and check that this
works.  First, note that, using that $g_k$ is asymptotic hyperbolic
with the conformal factor $r_k$, it follows that $g_{k+1}$ has the
same property (with the choice of $r_{k+1}=e^{\omega}r_k$). The proof
is identical with the case of the minimal induction step. See
\eqref{needed}.

Next, we will require the relation between the two
derivatives of the conformal functions.

With $\omega=-\frac{r_k^{k+1}F_k}{k+1}$, we have
$e^\om=1+r^{k+1}_k E_{k}$, where $E_{k}$ is a function smooth up
to the boundary, and so
\begin{align*}
  r_{k+1}&=e^\omega r_k=r_k(1+r^{k+1}_k {E}_k)=r_k+r^{k+2}_k{E}_k,
\end{align*}
and whence
\begin{align*}
  d r_{k+1}&=d r_k + (k+2)d r_k r^{k+1}_k {E}_k + r^{k+2}_k d {E}_k=dr_k+ O(r_k^{k+1}).
\end{align*}
This gives
\begin{align*}
  |d r_{k+1}|^2_{g_{k+1}}&=e^{-2\omega}|d r_k|_{g_k}^2+\mathcal{O}(r^{k+1}_k).
\end{align*}
And using $e^{-2\om}=1+O(r_k^{k+1})$ gives
\begin{align*}
  |d r_{k+1}|^2_{g_{k+1}}&=|dr_k|_{g_k}^2+\mathcal{O}(r^{k+1}_k),
\end{align*}
and so finally
\begin{align}\label{needed2}
\frac{|dr_k|_{g_k}^2}{|d r_{k+1}|^2_{g_{k+1}}}&=1+ \mathcal{O}(r^{k+1}_k).
\end{align}

With this choice of $\omega$, and these results, we now return to
equation \eqref{exp-min-ver-expanded} (for the induction step at hand
that is setting $\bar{g}:=g_{k+1}$ and $g:=g_k$) to verify that on the right
hand side everything vanishes, except
$r_{k+1}+\mathcal{O}(r_k^{k+1})$.

We first observe that
\begin{align*}
r_k^c \omega_c=-r_k^kF_k+r^{k+1}_k\tilde{F}_k,
\end{align*}
for some function $\tilde{F}_k$ that is smooth up to the boundary, and
can depend on $F_k$.  Now beginning with
\eqref{exp-min-ver-expanded}, we divide both sides by the $(n-1)|d
r_{k+1}|^3_{g_{k+1}}$ to yield
\begin{align*}
H_{r_{k+1}}e^{\omega}&=\frac{|dr_k|_{g_k}^3}{|d r_{k+1}|^3_{g_{k+1}}}H_{r_k}+\frac{|dr_k|_{g_k}^2}{|d r_{k+1}|^3_{g_{k+1}}} r_k^c\omega_c\\
&+\frac{1}{|d r_{k+1}|^3_{g_{k+1}}}r_k\mathcal{O}(r^{k}_k),
\end{align*}
where we have reduced the last four lines of the
\eqref{exp-min-ver-expanded} to $r_k \mathcal{O}(r_k^k)$ using that $\omega=\mathcal{O}(r^{k+1}_k)$ and an almost identical
computation as in the minimal case. Furthermore this last term in the
display here is overall of order $\mathcal{O}(r^{k+1}_k)$, as
$\frac{1}{|d r_{k+1}|^3_{g_{k+1}}}=1+\mathcal{O}(r_k)$.

We continue by substituting $H_{r_k}=r_k+r^k_kF_k$, also
$r_k^c \omega_c=-r_k^kF_k+r^{k+1}_k\tilde{F}_k$, as computed above,
$e^{\omega}=1+\mathcal{O}(r^{k+1}_k)$, and  use  that, from \eqref{needed2},
$\frac{|dr_k|_{g_k}^2}{|d r_{k+1}|^3_{g_{k+1}}}=1+r_kG_k$,
for
some function $G_k$ that is
smooth up to the boundary. We obtain
\begin{align*}
  H_{r_{k+1}}\big(1 &+  \mathcal{O}(r^{k+1}_k)\big)\\
  &=\frac{|dr_k|_{g_k}^3}{|d r_{k+1}|^3_{g_{k+1}}}\big(r_k+r_k^kF_k\big)+\big(1+r_kG_k\big)\big( -r_k^kF_k+r^{k+1}_k\tilde{F}_k\big)+\mathcal{O}(r^{k+1}_k),\\
&=\big(1+\mathcal{O}(r^{k+1}_k)\big)\big(r_k+r_k^kF_k\big)+\big(1+r_kG_k\big)\big( -r_k^kF_k+r^{k+1}_k\tilde{F}_k\big)
+\mathcal{O}(r^{k+1}_k),\\
&=r_k+\mathcal{O}(r^{k+1}_k)\\
&=r_{k+1}+\mathcal{O}(r_{k+1}^{k+1}),
\end{align*}
where we have used that \eqref{needed2} to the power of $3/2$ gives the leading behaviour claimed in the second line.

This completes our last step of the induction.

\end{proof}

\section*{acknowledgements}
Authors are grateful to Rafe Mazzeo, Robin Graham, Pedram Hekmati, Ben Andrews and Nicolau Sarquis Aiex for insightful discussions on the topic.
Both authors gratefully acknowledge support from the Royal Society of
New Zealand via Marsden Grants 16-UOA-051 and 19-UOA-008, and from the
Australian Research Council through DE190100379 and DP180100431.

\nocite{fefferman2012ambient}
\nocite{frauendiener2004conformal}
\nocite{gover2014boundary}
\nocite{bailey1994thomas}
\nocite{lawson1974foliations}
\nocite{frobenius1877ueber}
\nocite{vcap2016projective}
\nocite{vcap2016projective2}
\nocite{rummler1979quelques}

\bibliographystyle{plain}
\bibliography{mbib}

\end{document}